\documentclass[10pt]{amsart}

\usepackage{amssymb,amsmath,amscd,amsthm,amsfonts,wasysym,mathrsfs,amsxtra,multirow}
\input xy
\xyoption{all}

\newcommand{\RR}{\mathbb{R}}
\newcommand{\OO}{\mathscr{O}}

\newcommand{\Cfield}{\mathbb{C}}

\newcommand{\Spec}{\textnormal{Spec}\,}

\newcommand{\HHom}{\mathscr{H}om} 
\newcommand{\Hom}{\textnormal{Hom}}
\newcommand{\dimension}{\textnormal{dim}\,}

\newcommand{\rank}{\textnormal{rk}\,}

\newcommand{\Ext}{\textnormal{Ext}}
\newcommand{\EExt}{\mathscr{E}xt}

\newcommand{\al}{\alpha}

\newcommand{\Coh}{\textnormal{Coh}}

\newcommand{\Ap}{\mathcal{A}^p}

\newcommand{\arinj}{\ar@{^{(}->}}
\newcommand{\arsurj}{\ar@{->>}}
\newcommand{\areq}{\ar@{=}}

\newtheorem{theorem}{Theorem}[section]
\newtheorem{lemma}[theorem]{Lemma}
\newtheorem{coro}[theorem]{Corollary}
\newtheorem{pro}[theorem]{Proposition}

\theoremstyle{definition}

\theoremstyle{remark}
\newtheorem{remark}[theorem]{Remark}

\begin{document}

\title{Polynomial Bridgeland Stable Objects and Reflexive Sheaves}

\author{Jason Lo}
\address{Department of Mathematics \\ University of Missouri-Columbia \\ Columbia MO 65211, USA}
\email{locc@missouri.edu}



\begin{abstract}
On a smooth projective threefold $X$, we show that there are only two isomorphism types for the moduli of stable objects with respect to Bayer's standard polynomial Bridgeland stability - the moduli of Gieseker-stable sheaves and the moduli of PT-stable objects (see \cite{Lo2}) - under the following assumptions: no two of the stability vectors are collinear, and the degree and rank of the objects are relatively prime.  We also interpret the intersection of the moduli spaces of PT-stable and dual-PT-stable objects as a moduli of reflexive sheaves, and point out its connections with the existence problem of Bridgeland stability conditions on smooth projective threefolds, and the existence of fine moduli spaces of complexes on elliptic threefolds.
\end{abstract}

\maketitle

\section{Introduction}

For a few years after Bridgeland introduced his notion of stability conditions on triangulated categories in \cite{Bridgeland}, it was not known how stability conditions on the derived category of coherent sheaves $D^b(X)$ of a smooth projective threefold $X$ can be constructed in general.  Recently, Bayer-Macr\`{i}-Toda described a conjectural construction of a Bridgeland stability on arbitrary smooth projective threefolds in \cite{BMT}, which was verified for $X=\mathbb{P}^3$ by Macr\`{i} \cite{Macri}.  Once the moduli spaces of Bridgeland-semistable objects are constructed, they could be used to define invariants for the underlying threefold, using tools such as Behrend's constructible functions, or integration over virtual fundamental classes.  At this stage, however, it is not clear what the Bridgeland-semistable objects and their moduli, with respect to Bayer-Macr\'{i}-Toda's stability, look like on smooth projective threefolds in general.

Before the work \cite{BMT} appeared, Bayer defined the notion of polynomial stability in \cite{BayerPBSC} as an approximation of Bridgeland stability, and wrote down a `standard family' of polynomial stability conditions on any smooth projective variety.  (Note: Toda also defined a notion of limit stability in \cite{TodaLSOp}, which can be regarded as a type of polynomial stability.)  In particular, on a smooth projective threefold $X$, Bayer singled out two polynomial stabilities, which he called DT-stability and PT-stability.  He showed that the DT-stable objects of rank 1 and degree 0  are exactly the ideal sheaves of 1-dimensional subschemes of $X$, while the PT-stabe objects of rank 1 and degree 0 are exactly the 2-term complexes given by stable pairs studied in Pandharipande-Thomas \cite{PT}.  Besides, DT-stability and PT-stability are related by a wall-crossing in the space of polynomial stability conditions.    Therefore, polynomial stability gives a viewpoint for higher-rank analogues of stable pairs, potentially helping us understand higher-rank Donaldson-Thomas (DT) invariants.

The moduli spaces of PT-semistable objects were constructed as universally closed algebraic spaces of finite type in \cite{Lo1, Lo2}.  One motivation for this article is to understand other moduli spaces that could arise from polynomial stabilities on threefolds.  As it turns out, under a mild assumption on the parameters for polynomial stability (condition V1 in Section \ref{section-polystab3folds} below), there is only one more type of moduli spaces  other than the moduli spaces of DT-semistable objects and the moduli spaces of PT-semistable objects - this follows from the discussion in Section \ref{section-polystab3folds}.  Furthermore, if we only consider objects that have relatively prime degree and rank, then the moduli spaces of DT-semistable objects, which are the moduli of Gieseker-semistable sheaves, and the moduli spaces of PT-semistable objects are the only moduli spaces that can arise (Theorem \ref{pro-onlytwomoduli}).

In the process of proving Theorem \ref{pro-onlytwomoduli}, we obtain homological characterisations of semistable objects with respect to various polynomial stabilities.  A by-product of this is an algebraic space of finite-type that parametrises 2-term complexes $E^\bullet$ on a smooth projective threefold $X$ such that $H^{-1}(E)$ is a reflexive sheaf, and $H^0(E)$ is a 0-dimensional sheaf (Theorem \ref{theorem-intersectoftwopropermoduli}).  This moduli space can be considered as a moduli space of reflexive sheaves, where a reflexive sheaf may be `decorated' with extra points lying on its singularity locus (which is a codimension-3 locus).  Aside from this, this moduli space is interesting in its own right for the following two reasons:

First, the 2-term complexes described above resemble a particular class of `tilt-semitable objects' defined in Bayer-Macr\`{i}-Toda (see \cite[Section 7.2]{BMT}).  In particular, they show that the existence of Bridgeland stability conditions on a threefold is equivalent to a Bogomolov-Gieseker-type inequality for tilt-stable objects of slope 0 \cite[Conjecture 3.2.7]{BMT}.  If we can understand the relations between the objects parametrised by the moduli space in Theorem \ref{theorem-intersectoftwopropermoduli} and tilt-stable objects in the sense of Bayer-Macr\`{i}-Toda, we can hope to make progress towards \cite[Conjecture 3.2.7]{BMT} and hence the existence of Bridgeland stabilities on threefolds.    At the very least, we can expect to produce more examples of objects satisfying their conjectural inequality, by using existing results on stable reflexive sheaves on threefolds in works such as Hartshorne's \cite{SRS}, Langer's \cite{Langer} and Mir\'{o}-Roig's \cite{MR}.

Second, the moduli space of complexes in Theorem \ref{theorem-intersectoftwopropermoduli} gives an example of a moduli of stable complexes that is a fine moduli space.  In a forthcoming article by the author \cite{Lo3}, we study Fourier-Mukai transforms on elliptic threefolds, and identify a criterion under which 2-term complexes are mapped to torsion-free sheaves via the Fourier-Mukai transforms constructed by Bridgeland-Maciocia \cite{BMef}.  We show that each of these Fourier-Mukai transforms induces an open immersion from an open subspace $\mathcal{N}$ of the moduli space of complexes in Theorem \ref{theorem-intersectoftwopropermoduli} to a moduli of Gieseker-stable torsion-free sheaves.  That is, $\mathcal{N}$ is  a fine moduli space of complexes.

\subsection{Statements of main results}

To define a  polynomial stability on a smooth projective variety $X$ as in \cite{BayerPBSC}, we need to choose stability vectors $\rho_i$ (where $0\leq i\leq \dimension X$), which are nonzero complex numbers, and a perversity function $p$ on the topological space of $X$ that is compatible with the $\rho_i$.  The perversity function $p$ determines the heart $\Ap$ of a bounded t-structure on $D^b(X)$.  Given a polynomial stability $\sigma$, we can fix a Chern character $ch$, and ask whether we can construct the moduli stack parametrising $\sigma$-semistable objects of Chern character $ch$ in $\Ap$.  Our first main result is the following:
\begin{theorem}\label{pro-onlytwomoduli}
Let $X$ be a smooth projective threefold.  Let $C$ be the set of all possible Chern characters $ch$ such that $ch_0$ is nonzero, and $ch_0, ch_1$ are relatively prime.  Let $\sigma$ be any polynomial stability condition on $D^b(X)$ where no two of the stability vectors $\rho_i$ are collinear.  Let $\mathscr{M}^\sigma_{ch}$ denote the moduli space of $\sigma$-stable objects of Chern character $ch$.  Then for any $ch \in C$, the moduli space $\mathscr{M}^\sigma_{ch}$ is isomorphic to one of the spaces in the following lists:
\begin{itemize}
\item The  moduli spaces of Gieseker-stable sheaves of Chern character $ch$, where $ch \in C$.
\item The moduli spaces of PT-stable objects of Chern character $ch$ in $\Ap$, where $p(d) = -\lfloor \frac{d}{2}\rfloor$ and $ch \in C$.
\end{itemize}
\end{theorem}
 When $p(d) = -\lfloor \frac{d}{2}\rfloor$ as in this proposition, the heart of bounded t-structure $\Ap \subset D^b(X)$ is given by $\Ap = \langle \Coh_{\leq 1}(X), \Coh_{\geq 2}(X)[1] \rangle$.  Here,  $\Coh_{\leq 1}(X)$ denotes the category of coherent sheaves on $X$ whose support have dimension at most 1, $\Coh_{\geq 2}(X)$ denotes the category of coherent sheaves that do not have torsion subsheaves supported in dimension 1 or less, and $\langle \Coh_{\leq 1}(X), \Coh_{\geq 2}(X)[1] \rangle$ denotes the smallest extension-closed subcategory of $D^b(X)$ containing $\Coh_{\leq 1}(X)$ and $\Coh_{\geq 2}(X)[1]$.

Our second main result is the following:
\begin{theorem}\label{theorem-intersectoftwopropermoduli}
Let $X$ be a smooth projective threefold.  Let $ch$ be a Chern character such that $ch_0 \neq 0$, and $ch_0, ch_1$ are relatively prime.  Then there is an algebraic space of finite type, which is the intersection of two proper algebraic spaces of finite type, parametrising all objects $E \in \Ap$ of Chern character $ch$ such that $H^{-1}(E)$ is a $\mu$-stable reflexive sheaf, $H^0(E)$ is a 0-dimensional sheaf, and the map
    \[
      H^2 (\delta) :  \EExt^1 (H^{-1}(E),\OO_X) \to \EExt^3 (H^0(E),\OO_X)
    \]
(where $\delta$ is as in  \eqref{triangle-Edualised} below) is surjective.
\end{theorem}
In Section \ref{subsection-functorialreflexive}, we explain how the algebraic space in this theorem can be seen as a functorial construction of the moduli of reflexive sheaves $F$ on $X$, where each isomorphism class $[F]$ occurs with multiplicity up to the number of distinct quotient sheaves (up to isomorphism) of $\EExt^1 (F,\OO_X)$.

These two main results follow naturally, once we have the homological characterisations of polynomial stable objects in Section \ref{section-stabobjchar}.

\subsection{Notation}

For a coherent sheaf $E$ on a scheme $X$, we write $E^\ast$ to denote the sheaf dual $\HHom (E,\OO_X)$; if $E \in D^b(X)$ is a complex of coherent sheaves on $X$, we write $E^\vee$ to denote the derived dual $R\HHom (E,\OO_X)$, and write $H^i(E)$ to denote the degree-$i$ cohomology (which is a coherent sheaf) of $E$.  We will use $\mathbb{D}(-)$ to denote the dualizing functor $(-)^\vee [2]$ on $D^b(X)$.

For a polynomial stability $\sigma$ on $D^b(X)$, we write $\sigma^\ast$ to denote the dual polynomial stability.  We will use $\Coh (X)$ to denote the category of coherent sheaves on $X$.  For any integer $d$, we write $\Coh_{\leq d}(X)$ to denote the category of coherent sheaves on $X$ whose support have dimension at most $d$, and write $\Coh_{\geq d}(X)$ to denote the category of coherent sheaves on $X$ that have no subsheaves supported in  dimension $d-1$ or less.  For any $0\leq d< d'\leq 3$, we will write $\langle \Coh_{\leq d}(X),\Coh_{\geq d'}(X)[1] \rangle$ to denote the smallest extension-closed subcategory of $D^b(X)$ containing $\Coh_{\leq d}(X)$ and $\Coh_{\geq d'}(X)[1]$.

\section{Polynomial Stabilities on threefolds}\label{section-polystab3folds}

Throughout this article,  $X$ will be a smooth projective threefold.

Consider a standard polynomial stability $\sigma = (\omega, \rho, p, U)$ in the sense of Bayer \cite{BayerPBSC}.  Recall that, here, $\omega$ is a fixed ample $\mathbb{R}$-divisor on $X$, whereas
\[
\rho = (\rho_0, \rho_1, \rho_2, \rho_3) \in (\Cfield^\ast)^{4}
 \]
 is a quadruple of nonzero complex numbers such that each $\rho_d/\rho_{d+1}$ lies in the upper half complex plane.  And $p$ is a perversity function associated to $\rho$, i.e.\ $p$ is a function $\{0,1,2,3\} \to \mathbb{Z}$ such that $(-1)^{p(d)}\rho_d$ lies in the upper half plane for each $d$.  The last part, $U$, of the data $\sigma$ is a unipotent operator (i.e.\ an element of $A^\ast (X)_\Cfield$ of the form $U=1+N$, where $N$ is concentrated in positive degrees).  The perversity function $p$ determines a t-structure on $D^b(X)$ with heart $\Ap$.  Once the data $\sigma$ is given, the group homomorphism (usually called the `central charge')
 \begin{align*}
   Z_\sigma :  K(D^b(X)) &\to \Cfield [m] \\
     E &\mapsto Z_\sigma(E)(m) := \int_X \sum_{d=0}^3 \rho_d \omega^d m^d ch(E) \cdot U
 \end{align*}
 has the property that $Z_\sigma(E)(m)$ lies in the upper half plane for any $0 \neq E \in \Ap$ and real number $m \gg 0$.

 For $0\neq E \in \Ap$, if we write $Z_\sigma(E)(m) \in \mathbb{R}_{>0} \cdot e^{i \pi \phi (E) (m)}$ for some real number $\phi (E)(m)$ for $m \gg 0$, then we have $\phi (E)(m) \in (0, 1]$ for $m \gg 0$.  We say that $E$ is $\sigma$-semistable if, for all subobjects $0 \neq F \subsetneq E$ in $\Ap$, we have $\phi (F)(m) \leq \phi (E) (m)$ for all $m \gg 0$ (which we write $\phi (F) \preceq \phi (E)$ to denote);  and we say $E$ is $\sigma$-stable if $\phi (F) (m) < \phi (E)(m)$ for all $m \gg 0$ (which we write $\phi (F) \prec \phi (E)$ to denote). The reader may consult \cite[Section 3.2]{BayerPBSC} for more details on the basics of polynomial stability.

Up to shifting the $\sigma$-semistable objects in $D^b(X)$, we may assume that $p(0)=0$.  Since the perversity function $p$ satisfies $p(d) \geq p(d+1) \geq p(d)-1$, there are only four such perversity functions that take on at least three distinct values, listed in Table \ref{eqn-fourstabfctns}.

Recall that for the dual stability $\sigma^\ast = (\omega, \rho^\ast, \bar{p}, U^\ast)$,  we use the dual perversity function $\bar{p}$ defined by $\bar{p}(d) = -d-p(d)$ \cite[Definition 3.1.1]{BayerPBSC}.  The duals of the four perversity functions in \eqref{eqn-fourstabfctns} all take on at most two distinct values.  Therefore, up to shifting and taking derived duals of the semistable objects, we  obtain all possible isomorphism classes of  moduli of semistable objects with respect to standard polynomial stabilities under the following assumption:
\begin{itemize}
\item[V0.] The perversity function $p$ takes on at most two distinct values, and  $p(0)=0$.
\end{itemize}
  This assumption implies that the heart of t-structure $\Ap$ is  of the form
  \[
  \Ap = \langle \Coh_{\leq d}(X), \Coh_{\geq d+1}(X)[1]\rangle
   \]
   for some $0\leq d < 3$; that is, it is obtained from $\Coh (X)$ by tilting once.

\begin{table}
\begin{tabular}[b]{c|c c c c}
$d$ & 0 & 1 & 2 & 3 \\
\hline
\multirow{4}{*}{$p(d)$} & 0 & 0 & $-1$ & $-2$ \\
& 0 & $-1$ & $-1$ & $-2$ \\
& 0 & $-1$ & $-2$ & $-2$ \\
& 0 & $-1$ & $-2$ & $-3$
\end{tabular}
\caption{Perversity functions $p$ with $p(0)=0$  that take on at least three distinct values.}
\label{eqn-fourstabfctns}
\end{table}

On the space of polynomial stabilities on $X$, we also have a $\widetilde{GL}^+(2,\RR)$-action \cite[see Lemma 8.2]{Bridgeland}, which does not alter the semistable objects.  Up to this action, there are only five distinct standard polynomial stabilities on $X$ satisfying V0 and the following condition:
\begin{itemize}
\item[V1.] No two of the stability vectors $\rho_i$ are collinear.
\end{itemize}
 These five polynomial stabilities correspond to the  configurations of stability vectors $\rho_i$ in Figure \ref{figure-PTstab} below, which we label as DT, PT, $\sigma_3$, $\sigma_4$ and $\sigma_5$.  They all have the same perversity function $p(d)=-\lfloor \frac{d}{2}\rfloor$.

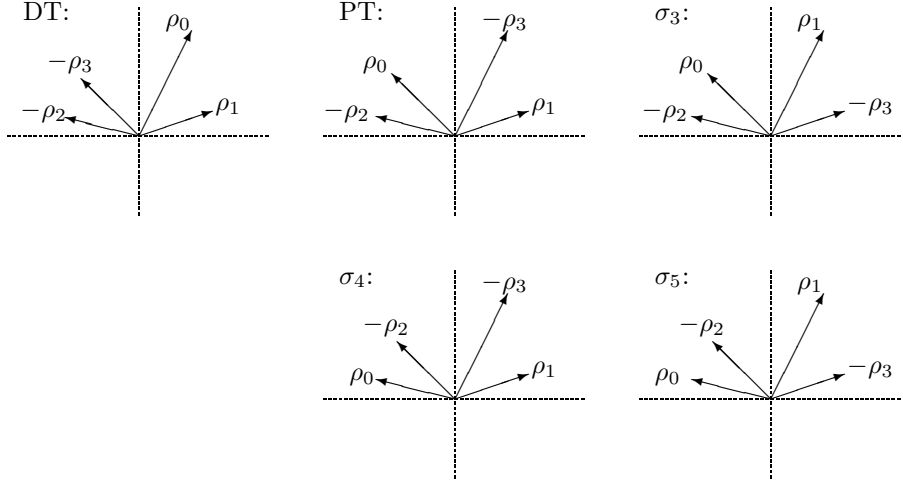
\begin{figure*}[h]
\centering
\setlength{\unitlength}{0.7mm}
\begin{picture}(170,90) 
\multiput(0,65)(1,0){50}{\line(1,0){0.5}}
\multiput(25,50)(0,1){40}{\line(0,1){0.5}}
\put(25,65){\vector(-4,1){14}}
\put(2.5,68.6){$-\rho_2$}
\put(25,65){\vector(-1,1){11}}
\put(7.5,78){$-\rho_3$}
\put(25,65){\vector(1,2){10}}
\put(30,86){$\rho_0$}
\put(25,65){\vector(3,1){14}}
\put(39.5,70){$\rho_1$}
\put(3,87){DT:}

\multiput(60,65)(1,0){50}{\line(1,0){0.5}}
\multiput(85,50)(0,1){40}{\line(0,1){0.5}}
\put(85,65){\vector(-4,1){15}}
\put(60,68.6){$-\rho_2$}
\put(85,65){\vector(-1,1){12}}
\put(67.5,78){$\rho_0$}
\put(85,65){\vector(1,2){10}}
\put(90,86){$-\rho_3$}
\put(85,65){\vector(3,1){14}}
\put(99.5,70){$\rho_1$}
\put(63,87){PT:}

\multiput(120,65)(1,0){50}{\line(1,0){0.5}}
\multiput(145,50)(0,1){40}{\line(0,1){0.5}}
\put(145,65){\vector(-4,1){15}}
\put(120.5,68.6){$-\rho_2$}
\put(145,65){\vector(-1,1){12}}
\put(127.5,78){$\rho_0$}
\put(145,65){\vector(1,2){10}}
\put(150,86){$\rho_1$}
\put(145,65){\vector(3,1){14}}
\put(159.5,70){$-\rho_3$}
\put(123,87){$\sigma_3$:}

\multiput(60,15)(1,0){50}{\line(1,0){0.5}}
\multiput(85,0)(0,1){40}{\line(0,1){0.5}}
\put(85,15){\vector(-4,1){15}}
\put(65,18.6){$\rho_0$}
\put(85,15){\vector(-1,1){11}}
\put(67.5,28){$-\rho_2$}
\put(85,15){\vector(1,2){10}}
\put(90,36){$-\rho_3$}
\put(85,15){\vector(3,1){14}}
\put(99.5,20){$\rho_1$}
\put(63,37){$\sigma_4$:}

\multiput(120,15)(1,0){50}{\line(1,0){0.5}}
\multiput(145,0)(0,1){40}{\line(0,1){0.5}}
\put(145,15){\vector(-4,1){15}}
\put(123,18.6){$\rho_0$}
\put(145,15){\vector(-1,1){11}}
\put(127.5,28){$-\rho_2$}
\put(145,15){\vector(1,2){10}}
\put(150,36){$\rho_1$}
\put(145,15){\vector(3,1){14}}
\put(159.5,20){$-\rho_3$}
\put(123,37){$\sigma_5$:}
\end{picture}

\caption{Configurations of the $\rho_i$ for five polynomial stabilities on a threefold}
\label{figure-PTstab}
\end{figure*}

Note that the dual stability vectors for PT differ from the stability vectors of $\sigma_5$ by a rotation of the complex plane (i.e.\ a $\widetilde{GL}^+(2,\RR)$-action), as is the case for $\sigma_4$ and $\sigma_3$. As a consequence, the semistable objects with respect to PT-stability are dual to those with respect to  $\sigma_5$-stability up to shift, and similarly for  $\sigma_3$-stability and $\sigma_4$-stability.  Overall, up to shifting and taking derived duals of the semistable objects, there are only three distinct moduli spaces (up to isomorphism) for standard polynomial stabilities whose stability vectors $\rho_i$ satisfy condition V1 above, given by DT, PT and $\sigma_3$-stabilities.

In Section \ref{subsection-V2char}, we will show (Corollary \ref{coro-PT4equiv}) that under a coprime assumption on degree and rank, PT-stability and $\sigma_4$-stability are equivalent.  This implies that there are only two distinct moduli of semistable objects with respect to standard polynomial stabilities, up to taking derived dual: the moduli of DT-stable objects, and the moduli of PT-stable objects.

For convenience, let us introduce two more conditions on the stability vectors $\rho_i$ of a polynomial stability on a threefold below.  For a complex number $\rho$ lying on the upper half plane, let $\phi (\rho) \in (0,1]$ denote its phase.
\begin{itemize}
\item[V2.] The perversity function is $p(d)=-\lfloor \frac{d}{2} \rfloor$ (this satisfies  condition V0), condition V1 is satisfied, the phases  $\phi (\rho_0), \phi (-\rho_2)$ are both larger than $\phi (-\rho_3)$, and $\phi (-\rho_3) > \phi (\rho_1)$.
\item[V3.] The perversity function is $p(d)=-\lfloor \frac{d}{2} \rfloor$ (this satisfies  condition V0), condition V1 is satisfied, and the phases $\phi (\rho_0), \phi (\rho_1), \phi (-\rho_2)$ are all larger than $\phi (-\rho_3)$.
\end{itemize}
 Note that V3 is equivalent to:
\begin{itemize}
\item[V3'.] The perversity function is $p(d)=-\lfloor \frac{d}{2} \rfloor$ (this satisfies  condition V0), condition V1 is satisfied, and the phases $\phi (\rho_0), \phi (-\rho_2)$ are both larger than $\phi (\rho_1)$, which is  larger than $\phi (-\rho_3)$.
\end{itemize}
Note that PT and $\sigma_4$-stabilities both satisfy condition V2, while $\sigma_3$ and $\sigma_5$-stabilities both satisfy condition V3.  On the other hand, for $p (d) = -\lfloor \frac{d}{2} \rfloor$, an object $E \in \Ap$ is semistable with respect to $\sigma$, where $\sigma$ satisfies condition V2 (resp.\ V3), if and only if $E$ is semistable with respect to PT or $\sigma_4$ (resp.\ $\sigma_3$ or $\sigma_5$).

\section{Characterising Polynomial Stable Objects}\label{section-stabobjchar}

\begin{remark}\label{rmk-observation1}
   Here is a simple observation (made in \cite{LiQin}, for instance): let $p=-\lfloor \frac{d}{2} \rfloor$ be as above, so that $\Ap = \langle \Coh_{\leq 1}(X),\Coh_{\geq 2}(X)[1]\rangle$.  Take any complex $E \in \Ap$ with torsion-free $H^{-1}(E)$, and suppose $T$ is the cokernel of the canonical map $H^{-1}(E) \to H^{-1}(E)^{\ast \ast}$.  Then we have the short exact sequence of coherent sheaves
\[
  0 \to H^{-1}(E) \to H^{-1}(E)^{\ast \ast} \to T \to 0.
\]
On a smooth projective threefold $X$, the dimension of the support of $T$ is at most 1, giving us the short exact sequence in $\Ap$
\begin{equation}\label{es1}
  0 \to T \to H^{-1}(E)[1] \to H^{-1}(E)^{\ast \ast}[1] \to 0.
\end{equation}
Since we also have the canonical short exact sequence in $\Ap$
\[
 0 \to H^{-1}(E)[1] \to E \to H^0(E) \to 0,
\]
we see that $T$ is a subobject of $E$ in $\Ap$.
\end{remark}

\subsection{Stable objects for stabilities satisfying V3}

\begin{lemma}\label{lemma-Cstablereflexive}
Suppose $\sigma$ is a  polynomial stability on $X$ satisfying condition $V3$.  If $E \in \Ap$ is a $\sigma$-semistable object of nonzero rank, then $H^{-1}(E)$ is a reflexive $\mu$-semistable sheaf.
\end{lemma}

\begin{proof}
Note that, since $\phi (-\rho_2) > \phi (-\rho_3)$, the $\sigma$-semistability of $E$ implies $H^{-1}(E)$ is torsion-free.  By Remark \ref{rmk-observation1},  if $T$ is a nonzero sheaf, then $T$ would destabilise $E$ in $\Ap$ since $\phi (\rho_0), \phi (\rho_1) > \phi (-\rho_3)$.  Hence $T$ must be zero, meaning $H^{-1}(E)$ is a reflexive sheaf.  That $H^{-1}(E)$ is $\mu$-semistable follows from $\phi (-\rho_2) > \phi (-\rho_3)$.
\end{proof}

\begin{lemma}\label{lemma-dualtrianglelesanal}
Suppose $\sigma$ is a  polynomial stability on $X$ satisfying condition $V3$.  Suppose also that $E \in \Ap$ is a $\sigma$-semistable object of nonzero rank, and $H^0(E)$ is a 0-dimensional sheaf.  Then $H^1(E^\vee) \cong H^{-1}(E)^\ast$, we have the short exact sequence of coherent sheaves
\begin{equation}\label{eq-lemma-dualtrianglelesanal-1}
  0 \to H^2(E^\vee) \to \EExt^1 (H^{-1}(E),\OO_X) \to \EExt^3 (H^0(E),\OO_X) \to 0,
\end{equation}
and $H^3(E^\vee)=0$.  In particular, if $H^{-1}(E)$ is locally free, then $H^0(E)$ vanishes, and $E\cong H^{-1}(E)[1]$.
\end{lemma}

\begin{proof}
Given any object $E \in \Ap$, we can dualize the canonical exact triangle
\begin{equation}\label{triangle-EApcanonical}
  H^{-1}(E)[1] \to E \to H^0(E) \to H^{-1}(E)[2]
\end{equation}
to obtain
\begin{equation}\label{triangle-Edualised}
  H^0(E)^\vee \to E^\vee \to H^{-1}(E)^\vee [-1] \overset{\delta}{\to} H^0(E)^\vee [1].
\end{equation}
Since $H^i(H^0(E)^\vee) \cong \EExt^i (H^0(E),\OO_X)$, and $H^0(E)$ is 0-dimensional, $H^i (H^0(E)^\vee)$ is nonzero only when $i=3$.  On the other hand, \[
H^i (H^{-1}(E)^\vee[-1]) \cong H^{i-1}(H^{-1}(E)^\vee)\cong \EExt^{i-1}(H^{-1}(E),\OO_X)
\]
 is zero whenever $i\neq 1,2$; this is because, by Lemma \ref{lemma-Cstablereflexive}, $H^{-1}(E)$ is reflexive, and hence has homological dimension at most 1.  The lemma would follow by taking the long exact sequence of cohomology of the exact triangle \eqref{triangle-Edualised}, provided that $H^3(E^\vee)=0$.  Note that the cohomology of $E^\vee$ is concentrated in degrees 1, 2 and 3, since it is an extension of $H^{-1}(E)^\vee[-1]$ by $H^0(E)^\vee$.

Since $\phi (\rho_0) > \phi (-\rho_3)$ for $\sigma$, for any closed point $x \in X$, we have the vanishing of $\Hom_{D^b(X)}(\OO_x, E)$ where $\OO_x$ denotes the skyscraper sheaf with value $k$ supported at $x$.  Thus $0=\Hom (E^\vee, \OO_x^\vee)=\Hom (E^\vee, \OO_x[-3])$ for any $x \in X$.  Since we observed that the highest-degree cohomology of $E^\vee$ is at degree 3, this implies $H^3(E^\vee)=0$, and the lemma follows.
\end{proof}

\begin{lemma}[Boundedness]\label{lemma-boundednesseveryone}
Let $ch$ be any fixed Chern character, and let $\sigma$ be any of the following five stabilities on $X$: DT, PT, $\sigma_3$, $\sigma_4$ or $\sigma_5$.  Then the set of $\sigma$-semistable objects in $\Ap$ with Chern character $ch$ is bounded.
\end{lemma}

\begin{proof}
Boundedness for DT-semistable objects is a classical result, while boundedness for PT-semistable objects was shown in \cite[Proposition 3.4]{Lo1}.  By taking dual, we have boundedness for $\sigma_5$-semistable objects as well.  The proof of \cite[Proposition 3.4]{Lo1} works for $\sigma_4$-semistable objects of nonzero rank without change; by taking dual, we also have boundedness for $\sigma_3$-semistable objects of nonzero rank.  On the other hand, $\sigma_3$-semistable objects of rank zero are Simpson-semistable sheaves, so we have boundedness for them; by taking dual, we have boundedness for $\sigma_4$-semistable objects of rank zero.
\end{proof}

\begin{lemma}\label{lemma-coprimecharCst}
Let $E \in \Ap$ be an object of nonzero rank, with relatively prime degree and rank.  Let $\sigma$ be a polynomial stability satisfying condition V3.  If $E$ satisfies the following conditions:
\begin{itemize}
\item $H^{-1}(E)$ is $\mu$-stable;
\item $\Hom_{D^b(X)}(\Coh_{\leq 1}(X),E)=0$,
\end{itemize}
then $E$ is $\sigma$-stable.
\end{lemma}

\begin{proof}
Take any short exact sequence $0 \to A \to E \to B \to 0$ in $\Ap$.  From this, we have the long exact sequence of cohomology
\[
0 \to H^{-1}(A) \to H^{-1}(E) \overset{\al}{\to} H^{-1}(B) \to H^0(A) \to H^0(E) \to H^0(B) \to 0.
\]
If $\rank (H^{-1}(A))=0$, then since $H^{-1}(E)$ is torsion-free, we have $H^{-1}(A)=0$, and so $A=H^0(A)$.  Then, the hypothesis that $\Hom (\Coh_{\leq 1}(X),E)=0$ implies $A=0$.   So let us suppose $\rank (H^{-1}(A)) \neq 0$.

If $\rank (H^{-1}(A)) < \rank (H^{-1}(E))$, then by the $\mu$-stability of $H^{-1}(E)$, we have $\phi (A) \prec \phi (E)$.  On the other hand, if $\rank (H^{-1}(A)) = \rank (H^{-1}(E))$, then we have $\rank (H^{-1}(B))=0$, which means either $H^{-1}(B)$ is 2-dimensional or it is zero.  If $H^{-1}(B)$ is a 2-dimensional sheaf, then we have $\phi (E) \prec \phi (B)$.  On the other hand, if $H^{-1}(B)=0$, then $B=H^0(B) \in \Coh_{\leq 1}(X)$, and we still have $\phi (E) \prec \phi (B)$.

Hence $E$ is $\sigma$-stable.
\end{proof}

\begin{coro}\label{coro-reflexivesareV2}
Let $\sigma$ be any polynomial stability on $X$ satisfying condition V3.  For any $\mu$-stable reflexive sheaf $F$ on $X$ with relatively prime degree and rank, $F[1]$ is $\sigma$-stable.
\end{coro}

\begin{proof}
By Lemma \ref{lemma-coprimecharCst}, it suffices to show that $\Hom_{D^b(X)} (\Coh_{\leq 1}(X),F[1])=0$.  Let $A \in \Coh_{\leq 1}(X)$.  Then $\Hom (A,F[1])\cong \Ext^1 (A,F) \cong \Ext^2 (F,A \otimes \omega_X)$, which vanishes because $F$ is reflexive (see \cite[Proposition 5]{PV}).  Hence $F[1]$ is $\sigma$-stable.
\end{proof}

\begin{lemma}\label{lemma-nopure1dimsubobj}
Suppose an object $E \in \Ap$ of nonzero rank satisfies the following conditions:
\begin{enumerate}
\item $H^{-1}(E)$ is torsion-free and has homological dimension at most 1;
\item $H^0(E)$ is a 0-dimensional sheaf;
\item $H^2 (E^\vee)$ is 0-dimensional;
\item $H^3(E^\vee)=0$.
\end{enumerate}
Then $\Hom_{D^b(X)}(T,E)=0$ for any pure 1-dimensional sheaf $T$.
\end{lemma}

\begin{proof}
As in the proof of Lemma \ref{lemma-dualtrianglelesanal}, we can consider the exact triangle \eqref{triangle-Edualised} associated to $E$, and its long exact sequence of cohomology.      Since $H^0(E)$ is 0-dimensional by assumption, $H^i (H^0(E)^\vee) \cong \EExt^i (H^0(E),\OO_X)$ is nonzero only for $i=3$.  So $H^0(E)^\vee$ is a 0-dimensional sheaf sitting at degree 3.  On the other hand, $H^i (H^{-1}(E)^\vee [-1]) \cong \EExt^{i-1} (H^{-1}(E),\OO_X)$ is zero for $i-1 \geq 2$ (since $H^{-1}(E)$ has homological dimension at most 1 by assumption), and for $i-1 \leq -1$.  Hence $H^{-1}(E)^\vee [-1]$ is a complex with cohomology concentrated in degrees 1 and 2.  Since $E^\vee$ is an extension of $H^{-1}(E)^\vee [-1]$ by $H^0(E)^\vee$ in the derived category $D^b(X)$, $E^\vee$ itself has cohomology concentrated in degrees 1, 2 and 3. However, $H^3(E^\vee)=0$ by assumption, so the cohomology of $E^\vee$ is  concentrated in degrees 1 and 2.

For any pure 1-dimensional sheaf $T$, $H^i(T^\vee) \cong \EExt^i (T, \OO_X)$ is zero for $i \neq 2,3$.  So the cohomology of $T^\vee$ is concentrated in degrees 2 and 3.  Hence
\begin{align*}
  \Hom_{D^b(X)}(T,E) &\cong \Hom_{D^b(X)} (E^\vee, T^\vee) \\
  &\cong \Hom_{D^b(X)} (H^2(E^\vee), H^2(T^\vee)) \\
  & =0.
\end{align*}
where the last equality follows because $H^2(T^\vee) \cong \EExt^2 (T,\OO_X)$, the dual of $T$ in the sense of \cite[Definition 1.1.7]{HL}, is also pure 1-dimensional by \cite[Proposition 1.1.10]{HL}, and $H^2(E^\vee)$ is 0-dimensional by assumption.  Hence $\Hom (T,E)=0$ for any pure 1-dimensional sheaf.
\end{proof}

\begin{pro}\label{pro-CstH0zerodimcoprimechar}
Let $E \in \Ap$ be an object of nonzero rank, with relatively prime degree and rank, and such that $H^0(E)$ is 0-dimensional.  Let $\sigma$ be a polynomial stability on $X$ satisfying condition V3.  Then $E$ is  $\sigma$-semistable if and only if it is $\sigma$-stable, if and only if it satisfies all the following conditions:
\begin{enumerate}
\item[(a)] $H^{-1}(E)$ is torsion free with homological dimension at most 1;
\item[(b)] $H^{-1}(E)$ is $\mu$-stable;
\item[(c)] $E^\vee \in \langle \Coh_{\leq 0}(X),\Coh_{\geq 3}(X)[1]\rangle [-2]$.
\end{enumerate}
\end{pro}

\begin{proof}
Take any object $E \in \Ap$ with nonzero rank,  relatively prime degree and rank, and such that $H^0(E)$ is 0-dimensional.  Suppose $\sigma$ is any polynomial stability satisfying condition V3.  Suppose $E$ is $\sigma$-semistable.  By Lemma \ref{lemma-Cstablereflexive}, $E$ satisfies conditions (a) and (b).  Then, by Lemma \ref{lemma-dualtrianglelesanal}, the cohomology of $E^\vee$ is concentrated in degrees 1 and 2, and $H^1(E^\vee)$ is torsion-free (in fact, reflexive).  From the reflexivity of $H^{-1}(E)$, we get that $\EExt^1 (H^{-1}(E),\OO_X)$ is supported in dimensional 0; from the exact sequence \eqref{eq-lemma-dualtrianglelesanal-1}, we get that $H^2(E^\vee)$ is a 0-dimensional sheaf.  Hence $E$ satisfies condition (c).

Next, suppose $E \in \Ap$ has nonzero rank, with relatively prime degree and rank, that $H^0(E)$ is 0-dimensional, and $E$ satisfies conditions (a), (b) and (c).   By Lemma \ref{lemma-coprimecharCst}, $E$ would be $\sigma$-stable if we can show that $\Hom_{D^b(X)}(\Coh_{\leq 1}(X),E)=0$.

To prove $\Hom_{D^b(X)}(\Coh_{\leq 1}(X),E)=0$, we first show that $\Hom (\Coh_{\leq 0}(X),E)=0$.  To this end, it suffices to show  $\Hom_{D^b(X)} (\OO_x,E)=0$ where $\OO_x$ is the skyscraper sheaf supported at the closed point $x \in X$, for any $x$.  However, $\Hom (\OO_x, E) \cong \Hom (E^\vee, \OO_x^\vee) \cong \Hom (E^\vee, \OO_x [-3])=0$, because $H^3(E^\vee)=0$ by condition (c).  So $\Hom (\Coh_{\leq 0}(X),E)=0$ holds. Now, condition (c) also says that $H^2(E^\vee)$ is a 0-dimensional sheaf.  By Lemma \ref{lemma-nopure1dimsubobj}, we have $\Hom_{D^b(X)} (T,E)=0$ for any pure 1-dimensional sheaf $T$.  This, combined with $\Hom (\Coh_{\leq 0}(X),E)=0$, gives the vanishing $\Hom_{D^b(X)}(\Coh_{\leq 1}(X),E)=0$, completing the proof of the proposition.
\end{proof}

\begin{coro}\label{coro-35equiv}
Let $E \in \Ap$ be an object of nonzero rank, with relatively prime degree and rank, and such that $H^0(E)$ is 0-dimensional.  Then $E$ is $\sigma_3$-stable if and only if it is $\sigma_5$-stable.
\end{coro}

\subsection{Stable objects for stabilities satisfying V2}\label{subsection-V2char}

\begin{lemma}\label{lemma-PTsigma4charcoprime}
Let $E \in \Ap$ be an object of nonzero rank, with relatively prime degree and rank.  Let $\sigma$ be a polynomial stability on $X$ satisfying condition $V2$.  Then $E$ is $\sigma$-semistable if and only if it is $\sigma$-stable  if and only if it satisfies the following conditions:
\begin{enumerate}
\item $H^{-1}(E)$ is torsion-free and $\mu$-stable;
\item $H^0(E)$ is 0-dimensional;
\item $\Hom_{D^b(X)}(\Coh_{\leq 0}(X),E)=0$.
\end{enumerate}
\end{lemma}

\begin{proof}
Suppose $E$ is $\sigma$-semistable.  Then $H^{-1}(E)$ is $\mu$-semistable because $\phi (-\rho_2) > \phi (-\rho_3)$.  Hence $H^{-1}(E)$ is $\mu$-stable by our coprime assumption on degree and rank.  Property (3) follows from $\phi (\rho_0) > \phi (-\rho_3)$.  Since we have a canonical surjection $E \twoheadrightarrow H^0(E)$ in $\Ap$ and $\phi (-\rho_3) > \phi (\rho_1)$, we get that $H^0(E)$ cannot be 1-dimensional, and so must be 0-dimensional.  Hence $E$ satisfies properties (1) through (3).

The rest of the proof is identical to the proof for the PT case in \cite[Proposition 2.24]{Lo2}.
\end{proof}

\begin{coro}\label{coro-PT4equiv}
Let $E \in \Ap$ be an object of nonzero rank, with relatively prime degree and rank.  Then $E$ is PT-stable if and only if $E$ is $\sigma_4$-stable.
\end{coro}

\begin{coro}\label{coro-everybodyisV2}
 Let $\sigma$ be any polynomial stability on $X$ satisfying condition V3.  Let $E \in \Ap$ be any object of nonzero rank with relatively prime degree and rank.  If $E$ is $\sigma$-semistable and $H^0(E)$ is 0-dimensional, then $E$ is stable with respect to any polynomial stability satisfying condition V2.
\end{coro}

\begin{proof}
Since $\phi (-\rho_2) >  \phi (-\rho_3)$ for $\sigma$, $H^{-1}(E)$ must be torsion-free and $\mu$-semistable, hence $\mu$-stable, by the coprime assumption. Since $\phi (\rho_0) > \phi (-\rho_3)$ for $\sigma$, the $\sigma$-stability of $E$ gives $\Hom_{D^b(X)}(\Coh_{\leq 0}(X),E)=0$.  The corollary then follows  from  Lemma \ref{lemma-PTsigma4charcoprime}.
\end{proof}

\section{Moduli of polynomial stable objects and reflexive sheaves}

\subsection{Moduli spaces of polynomial stable objects}

\begin{pro}[Openness]\label{pro-standardpolystopenness}
Let $ch$ be a fixed Chern character such that $ch_0\neq 0$ and $ch_0, ch_1$ are relatively prime.  Let $S$ be a Noetherian scheme over the ground field $k$, and let $E_S \in D^b (X \times_{\Spec k} S)$ be a flat family of objects in $\Ap$ over $S$ whose fibres have Chern character $ch$.  Let $\sigma$ be any  polynomial stability on $X$ satisfying V1.    Suppose $s_0 \in S$ is a point such that $E_{s_0}$ is $\sigma$-stable.  Then there is an open subset $U \subset S$ containing $s_0$ such that for all $s \in U$, the fibre $E_s$ is $\sigma$-stable.  (That is, for flat families of objects in $\Ap$ of Chern character $ch$, being $\sigma$-stable is an open property.)
\end{pro}

\begin{proof}
By the discussion in Section \ref{section-polystab3folds}, it suffices to check this when $\sigma$ is one of the stabilities listed in Figure \ref{figure-PTstab}.

 Since being isomorphic to a sheaf is an open property for a flat family of complexes, and being $\mu$-semistable is an open property for a flat family of sheaves, the proposition holds for the DT case.

 Since PT is dual to $\sigma_5$, and $\sigma_3$ is dual to $\sigma_4$, and taking derived dual preserves openness,  it suffices to consider the case when $\sigma$ is either PT or $\sigma_4$.  The proposition then follows from \cite[Proposition 2.24, Proposition 3.3]{Lo2} and Corollary \ref{coro-PT4equiv}.
\end{proof}

As a consequence, we obtain:

\begin{pro}\label{pro-everyoneisproper}
Let $ch$ be a Chern character such that $ch_0 \neq 0$, and $ch_0, ch_1$ are relatively prime.  Let $\sigma$ be any  polynomial stability on $X$ satisfying V1.  Then there is a proper algebraic space of finite type parametrising $\sigma$-stable objects of Chern character $ch$.
\end{pro}

\begin{proof}
 Using our result on openness (Proposition \ref{pro-standardpolystopenness}), we have an algebraic space parametrising the stable objects by the same argument as in \cite{TodaLSOp}.  Boundedness follows from Lemma \ref{lemma-boundednesseveryone}.  Separatedness follows from the same  argument as in \cite[Theorem 3.20]{TodaLSOp}.  As for universal closedness, by using dual if necessary, it suffice to check universal closedness for the moduli of PT-stable and the moduli of $\sigma_4$-stable objects.  However, these two moduli spaces coincide by Corollary \ref{coro-PT4equiv}.  And the universal closedness of the moduli of PT-stable objects follows from \cite[Theorem 2.23]{Lo2}.
\end{proof}

\begin{proof}[Proof of Theorem \ref{pro-onlytwomoduli}]
Since PT-stability is dual to $\sigma_5$-stability, and $\sigma_3$-stability is dual to $\sigma_4$-stability, the proposition follows from Corollary \ref{coro-PT4equiv} and the discussion in Section \ref{section-polystab3folds}.
\end{proof}

\subsection{An intersection of two moduli spaces}\label{subsection-functorialreflexive}

\begin{pro}\label{pro-dualintersectionstobj}
Given an object $E \in \Ap$ of nonzero rank and relatively prime degree and rank, the following are equivalent:
\begin{enumerate}
\item $E$ is stable with respect to a polynomial stability satisfying condition $V2$, as well as a polynomial stability satisfying condition $V3$.
\item $H^{-1}(E)$ is a $\mu$-stable reflexive sheaf, $H^0(E)$ is 0-dimensional, and the map
    \[
      H^2 (\delta) :  \EExt^1 (H^{-1}(E),\OO_X) \to \EExt^3(H^0(E),\OO_X)
    \]
where $\delta$ is as in the sequence \eqref{triangle-Edualised} is surjective.
\end{enumerate}
\end{pro}

\begin{proof}
Suppose condition (1) holds.  Then $H^{-1}(E)$ is reflexive by Lemma \ref{lemma-Cstablereflexive}, and $H^0(E)$ is 0-dimensional by Lemma \ref{lemma-PTsigma4charcoprime}.  Also, $H^3(E^\vee)=0$ by Proposition \ref{pro-CstH0zerodimcoprimechar}, implying $H^2(\delta)$ is surjective.  Hence condition (2) holds.

For the converse, suppose condition (2) holds.  From the long exact sequence of \eqref{triangle-Edualised}, we get that $H^3(E^\vee)=0$.  So by  Proposition \ref{pro-CstH0zerodimcoprimechar}, $E$ is stable with respect to any polynomial stability satisfying condition V3.  Then by Corollary \ref{coro-everybodyisV2}, $E$ is also stable with respect to any polynomial stability satisfying condition V2.  This completes the proof of the proposition.
\end{proof}

For any polynomial stability condition $\sigma$ satisfying V1, we now know $\sigma$-stability is an open property (by Proposition \ref{pro-standardpolystopenness}, under the coprime assumption).  Hence we can consider the moduli space parametrising objects that are both $\sigma$-stable and $\sigma^\ast$-stable, as in Theorem \ref{theorem-intersectoftwopropermoduli}.

\begin{proof}[Proof of Theorem \ref{theorem-intersectoftwopropermoduli}]
This follows from Propositions \ref{pro-everyoneisproper} and \ref{pro-dualintersectionstobj}.
\end{proof}

\begin{remark}
Of course, since we know PT-semistability is an open property from \cite{Lo2}, its dual, $\sigma_5$-semistability, is also an open property.  Then we have a moduli space parametrising complexes $E$ of nonzero rank that are both PT-semistable and $\sigma_5$-semistable, where $H^{-1}(E)$ is necessarily $\mu$-semistable and reflexive (Lemma \ref{lemma-Cstablereflexive}), and $H^0(E)$ is necessarily 0-dimensional (\cite[Lemma 3.3]{Lo1}), and the map $H^2(\delta)$ is surjective (Lemma \ref{lemma-dualtrianglelesanal}).  However, it is not clear that all complexes of this form are both PT-semistable and $\sigma_5$-semistable; Proposition \ref{pro-dualintersectionstobj} says that this is indeed the case under the coprime assumption on degree and rank.
\end{remark}

Now we explain why the algebraic space in Theorem \ref{theorem-intersectoftwopropermoduli} can be seen as a functorial construction of the moduli of reflexive sheaves.  Given any reflexive sheaf $F$ on a smooth projective threefold $X$, the sheaf $\EExt^1 (F,\OO_X)$ is 0-dimensional.  (When $X= \mathbb{P}^3$ and $F$ is a rank-two $\mu$-stable reflexive sheaf, for example, we have $c_3 (F) = h^0(\EExt^1 (F,\omega_X))$, in which case the length of $\EExt^1 (F,\OO_X)$ can be considered as the number of non-locally free points of $F$, counted with multiplicities \cite[Proposition 2.6]{SRS}.)  Let $q' : \EExt^1 (F,\OO_X) \twoheadrightarrow Q$ denote any nonzero quotient of $\EExt^1 (F,\OO_X)$ in $\Coh (X)$.    If we let $c$ denote the canonical map $F^\vee[1] \overset{c}{\to} H^0(F^\vee[1])$, then the composition
\[
q:  F^\vee[1] \overset{c}{\to} H^0(F^\vee[1]) = \EExt^1 (F,\OO_X) \overset{q'}{\to} Q
\]
is a nonzero morphism in $\Ap$.   Now, we have the string of isomorphisms
\begin{align*}
  \Hom_{D^b(X)} (F^\vee[1], Q) &\cong \Hom_{D^b(X)} (Q^\vee,F[-1]) \text{ by dualizing} \\
  &\cong \Hom_{D^b(X)} (Q^\vee [3],F[2]) \\
  &\cong \Ext^1_{D^b(X)}(Q^D,F[1]),
\end{align*}
where we write $Q^D$ to denote the only cohomology of $Q^\vee [3]$, which sits at degree 0 and is isomorphic to $\EExt^3 (Q,\OO_X)$.  Let $E$ be any complex representing a class in $\Ext^1 (Q^D,F[1])$ corresponding to $q$.  If we consider the exact triangle \eqref{triangle-Edualised} for $E$, then $H^2(\delta)=H^0(q)=q'$, which is surjective.  So if $F$ is a $\mu$-stable reflexive sheaf with relatively prime degree and rank, then $E$ would correspond to a point of the moduli space in Theorem \ref{theorem-intersectoftwopropermoduli}.

In summary, given any complex $E$  that represents a point of the moduli in Theorem \ref{theorem-intersectoftwopropermoduli}, $H^{-1}(E)$ is a $\mu$-stable reflexive sheaf of relatively prime degree and rank, and $(H^0(E))^D$ is a quotient sheaf of $\EExt^1 (H^{-1}(E),\OO_X)$.  Conversely, given any $\mu$-stable reflexive sheaf $F$ of relatively prime degree and rank, and any quotient sheaf $Q$ of $\EExt^{1}(F,\OO_X)$, we obtain a complex $E$ representing a point of the aforementioned moduli space, where $H^{-1}(E) \cong F$ and $H^0(E) \cong Q^D$.

 Let us write  $\mathscr{M}^{PT\cap PT^\ast}_{(ch_0,ch_1,ch_2,ch_3)}$ to denote the moduli
 space in Theorem \ref{theorem-intersectoftwopropermoduli} where the objects have Chern character $ch$.  For fixed $r, d, \beta$ where $r\neq0$ and $r, d$ are relatively prime, we can now consider the moduli functor
\begin{equation}\label{eqn-reflexivemodulic3free}
  \coprod_n \mathscr{M}^{PT \cap PT^\ast}_{(r,d,\beta,n)}
\end{equation}
 whose points correspond to all pairs of the form $([F],Q)$, where $[F]$ is the isomorphism class of a $\mu$-stable reflexive sheaf $F$ such that
 \[
 (ch_0(F[1]), ch_1(F[1]), ch_2(F[1])) =(r,d,\beta),
  \]
  and $Q^D$ is a quotient of $\EExt^1 (F,\OO_X)$.   Each reflexive sheaf $F$ occurs in \eqref{eqn-reflexivemodulic3free} as many times as there are quotients of the 0-dimensional sheaf $\EExt^1 (F,\OO_X)$.

 \section*{Acknowledgements}

This project was partially inspired by the work of Wei-Ping Li and Zhenbo Qin \cite{LiQin}.  The author would like to thank Zhenbo Qin  for helpful discussions, Yogesh More for pointing out a simplification in a proof, Calin Chindris for various comments, and the referee for pointing out an error in an earlier version and suggesting various improvements on the presentation.

\bibliographystyle{plain}
\bibliography{biblio}

\end{document}